\documentclass[twoside,leqno,12pt]{amsart}

\usepackage{amsmath}

\def\boldit#1{\textit{\textbf{#1}}}

\def\dotcup{\operatorname{\dot\cup}}

 \input xy
 \xyoption{all}
 \def\boldit#1{\textit{\textbf{#1}}}

\def\SigCyc{\mbox{$\Sigma$-cyc}}

\def\sup{\mbox{\rm sup}}

 \newtheorem{thm}{Theorem}[section]
 
 \newtheorem{lem}[thm]{Lemma}
 \newtheorem{prop}[thm]{Proposition}
 \theoremstyle{definition}
 
 \theoremstyle{remark}
 \newtheorem{rem}[thm]{Remark}
 
 \numberwithin{equation}{section}

\begin{document}

%
%
%
%
%
%
%
%
%

\font\bfit=cmmib11 scaled \magstep 1
\font\itmed=cmti11
\def\boldp{{\bfit p}}
\def\medp{{\itmed p}}
\title[Abelian \medp-groups with a fixed elementary subgroup...]
      {Abelian \boldp-groups\\
        with a fixed elementary subgroup\\
        or with a fixed elementary quotient}

\author[J.\ Kosakowska]{Justyna Kosakowska}

\address{%
  Mathematics and Computer Science, Nicolaus Copernicus University, Toru\'n, Poland
}
\email{justus@mat.umk.pl}


\author[M.\ Schmidmeier]{Markus Schmidmeier}
\address{Mathematics and Statistics, Florida Atlantic University,
              Boca Raton, Florida}
\email{markus@math.fau.edu}

\author[M.\ Schreiner]{Martin Schreiner}
\address{Woerthstr. 21, 81667 München, Germany}
\email{mwschreiner@gmx.net}



\date{January 1, 2004}

\begin{abstract}
In his 1934 paper, G.\ Birkhoff poses the problem of classifying pairs $(G,U)$
  where $G$ is an abelian group and $U\subset G$ a subgroup,
  up to automorphisms of $G$. In general, Birkhoff's Problem is not considered
  feasible.
  In this note, we fix a prime number $p$ and assume that $G$ is a direct sum of cyclic $p$-groups
  and $U\subset G$ is a~subgroup. Under the assumption that the factor group
  $G/U$ is an elementary abelian $p$-group,
  we show that the  pair $(G,U)$ always has a direct sum decomposition
  into pairs of type $(\mathbb Z/(p^n),\mathbb Z/(p^n))$ or $(\mathbb Z/(p^n), (p))$.
  Surprisingly, in the dual situation we need an additional condition.
  If we assume that $U$ itself is an elementary
  subgroup of $G$, then we show that
  the pair $(G,U)$ has a direct sum decomposition
  into pairs of type $(\mathbb Z/(p^n),0)$ or $(\mathbb Z/(p^n), (p^{n-1}))$
  if and only if $G/U$ is a~direct sum of cyclic $p$-groups.
  We generalize the above results to modules over commutative discrete valuation rings.
\end{abstract}

\maketitle

\section{Introduction}
For $\mathcal A$ a full subcategory of the category of abelian groups,
the \boldit{subgroup category} $\mathcal S(\mathcal A)$
consists of all pairs $(G,U)$ where $G\in\mathcal A$ and $U$ is a subgroup
of $G$.  A morphism in $\mathcal S(\mathcal A)$
from $(G,U)$ to $(G',U')$ is a group homomorphism
$f:G\to G'$ such that the image $f(U)$ is contained in $U'$.

  There has been a lot of recent interest in categories of type $\mathcal S(\mathcal A)$
  although the problem of classifying subgroup embeddings has been studied since G.\ Birkhoff \cite{b}
  in 1934.
  The survey paper \cite{k} provides an overview for a variety of classification results;
  it presents comments and references to the literature about more general monomorphism categories,
  about homological algebra (Gorenstein-projective modules), combinatorics (Littlewood-Richardson tableaux),
  geometry (weighted projective lines), and applications (topological data analysis); the paper also reviews
  and extends recent results linking subgroups of abelian groups to invariant subspaces of linear operators.

  This manuscript is about two subcategories of $\mathcal S(\mathcal A)$, one where each object $(G,U)$
  has a elementary subgroup $U$, and an other where the factor $G/U$ modulo the subgroup is elementary.
  Clearly, if $\mathcal A$ consists of all finite $p$-groups, then the two subcategories are dual to each other.
  Yet, to the surprise of the authors, this duality does no longer hold if, say, $\mathcal A$ consists of all
  abelian groups which are direct sums of cyclic (or finite) $p$-groups.
  It turns out that the latter subcategory is quite well-understood as each object is shown to be a direct sum
  of objects where the global group is cyclic,
  while the former is more complicated as it contains objects $(G,U)$ for which $G/U$ has
  higher Ulm invariants (hence $(G,U)$ cannot possibly be a direct sum of objects with global group cyclic).

  \medskip
    In this paper, we first state our results for abelian groups
    which are direct sums of cyclic $p$-groups, and then present the more general
    versions   for direct sums of cyclic modules over any discrete valuation domain.
    The proofs in Sections \ref{seq-proof-thm1} and \ref{sec-proof3} are for the
    second, more general case.

  \medskip
Recall that an abelian group is $\Sigma$-\boldit{cyclic} if it is a direct sum of cyclic groups.
In this note, we fix a prime number $p$ and consider the category $\mathcal A=\SigCyc$
of all abelian groups which are direct sums of cyclic $p$-groups.
  Note that the summands in each direct sum can be indexed by a
  (possibly infinite) cardinal number.  We will work within ZFC.

By $\mathcal F_1(\SigCyc)$ we denote the full subcategory of $\mathcal S(\SigCyc)$
of all pairs $(G,U)$ where $G\in\SigCyc$ and $U$ is a subgroup of $G$
for which the factor is an elementary abelian $p$-group,
  that is, $p\cdot(G/U)=0$.

Let $\mathcal{S}_1(\SigCyc)$ be the full subcategory of $\mathcal{S}(\SigCyc)$ of all pairs $(G,U)$,
where $G\in\SigCyc$ and $U$ itself is an elementary subgroup of $G$.

A \boldit{picket} is a pair $(G,U)$ in $\mathcal{S}(\SigCyc)$ where the group $G$ is a cyclic $p$-group. For example every picket in $\mathcal F_1(\SigCyc)$ is isomorphic to
$$P^n_n=(\mathbb Z/(p^n),\mathbb Z/(p^n))\qquad or \qquad P^n_{n-1}=(\mathbb Z/(p^n),(p)),$$
for some $n\in\mathbb N$, whereas every picket in $\mathcal{S}_1(\SigCyc)$ is  isomorphic to
$$P_0^n=(\mathbb Z/(p^n),0)\qquad or \qquad P^n_{1}=(\mathbb Z/(p^n),(p^{n-1})),$$
for some $n\in \mathbb N$.

We first state the result for $\mathcal F_1$.

\begin{thm}
  \label{theorem-Fone}
  Every object in $\mathcal F_1(\SigCyc)$ is isomorphic to a direct sum
  of pickets of type $P^n_n$ or $P^n_{n-1}$.  The direct sum decomposition is unique
  up to reordering and isomorphy of the summands.
\end{thm}

\begin{rem}
  \label{remark2}
  The corresponding result for pairs $(G,U)$ where the subgroup $U$ is
  elementary 
  (and not the factor $G/U$) does not hold:
  Let $G=\bigoplus_{n\in\mathbb N} y_n\, \mathbb Z/(p^n)$ be the direct sum of cyclic $p$-groups
  where $y_n$ denotes the generator of the $n$-th direct
  summand, and take for $U$ the elementary subgroup
  $U=\bigoplus_{n>1} (y_1-p^{n-1}y_n)\mathbb Z$.
  Then the coset of $y_1$ in the factor group $G/U$ is divisible by all powers of $p$,
  so $G/U$ cannot be a~direct sum of cyclic $p$-groups.
  A different way to look at it is that
  the quotient $G/U$ is a countable reduced primary group which has
  an Ulm invariant at $\omega$ \cite[Remark (b) on p.\ 31]{kap}.
  In particular, the pair $(G,U)$ is not a direct sum of
  pickets of type $P_0^n$ or $P_1^n$.
\end{rem}

Because of Remark \ref{remark2},
  the corresponding result for $\mathcal S_1$ requires an additional condition.

\begin{thm} An object $(G,U)$ in $\mathcal S_1(\SigCyc)$ is a direct sum of pickets
    of type $P_0^n$ or $P_1^n$ if and only if $G/U\in\SigCyc$.  In this case,
  the direct sum decomposition is unique up to reordering and isomorphy of the summands.
\label{thm-s1}
\end{thm}

 \begin{rem}
   Some comments about the semisimplicity condition for subgroups or factors:

   (1) The nonsplit short exact sequence
    $$0\longrightarrow \mathbb Z/(p^2)\stackrel{{\iota\choose\pi}}\longrightarrow \mathbb Z/(p^3)\oplus
    \mathbb Z/(p)\longrightarrow \mathbb Z/(p^2)\longrightarrow 0$$
    shows that the condition that the subgroup or the factor group be
    elementary is needed
    to decompose the embedding (or the short exact sequence) as a direct sum of pickets.

    (2) If we assume that the group $G$ is $p^m$-bounded for some $m\in\mathbb N$
    (that is, $p^m\cdot G=0$),
    and $U$ is a subgroup such that either $U$ or $G/U$ is
    elementary, then our results are well-known:
      There are only
    finitely many finitely generated indecomposable objects, up to isomorphy,
    (namely the pickets, see for example \cite[3.1]{S}),
    and hence by the Auslander and Ringel-Tachikawa theorem any (not necessarily finitely generated) object
    is a direct sum of those (see \cite{M} for the Auslander and Ringel-Tachikawa theorem for embeddings).

    (3)
    Note that under the assumption that $G$ is $p^m$-bounded,
    the condition that $U$ or $G/U$ be elementary
    can even be relaxed to $U$ or $G/U$ being $p^2$-bounded; in this case any object is a direct sum of
    certain bi-pickets (see \cite[3.2]{S}).
    However, a further relaxation of the bound on the subgroup is not possible:
      For fixed $m\geq 8$,
      the problem of classifying all pairs $(G,U)$ where $G$ is a $p^m$-bounded finite abelian group
      and $U$ a $p^3$-bounded subgroup is considered infeasible, see \cite[3.7]{S}, \cite{RS}.
 \end{rem}

        We now state our main results in the more general setting.
      Let $R$ be a~commutative discrete valuation ring with the unique maximal ideal $(p)$, where $p\in R$ is non-nilpotent. All modules in this section are $R$-modules.

      A~module $M$ is {\it primary} (or {\it torsion}), if every element $x\in M$
      is annihilated by some power of $p$ (that is, $p^k x=0$ for some $k\in\mathbb N$).
      A~submodule $N$ of  $M$ is {\it pure} in $M$, if $p^kN=N\cap p^kM$
for all $k\in \mathbb{N}$.  A~subset $X$ of $M$ is {\it pure} in $M$, if the submodule $\langle X\rangle$ of $M$ generated by $X$ is pure in $M$.
 
 An element $x\in M$ has {\it height} $h_M(x)=k$, if $x\in p^k M$ and $x\notin p^{k+1}M$ for some $k\geq 0$ and
 $x$ has {\it infinite height} if $x\in p^k M$ for every $k\geq0$. We call a~subset $X$ of $M$ {\it height-bounded}, if $\sup\{h_{M}(x)|x\in X\}$ is finite. We denote the socle of $M$ by $M[p]=\{x\in M\;|\; px=0\}$. One can show that $N$ is pure in $M$ is and only if $h_N(x)=h_M(x)$ for all $x\in N[p]$, see \cite[Exercise 6, page 5036]{kt}.  For more information on discrete valuation rings and their modules see \cite{kt}.

By $\SigCyc(R)$ we denote the full subcategory of $R$-Mod consisting of $R$-modules
that are direct sums of cyclic primary $R$-modules.
Then  $\mathcal{S}_1(\SigCyc(R))$ and $\mathcal F_1(\SigCyc(R))$
consist of all pairs $(M,U)$
with $M\in\SigCyc(R)$ and $U$ an~$R$-submodule of $M$ such that
$pU=0$ and $p(M/U)=0$ holds, respectively.
For $n\in\mathbb N$ and $0\leq \ell \leq n$,
we define the pickets $P_\ell^n=(R/(p^n),(p^{n-\ell}))$
as usual.

Abelian $p$-groups are just the primary modules over the localization
$\mathbb Z_{(p)}$ of the integers at the prime ideal $(p)$, which is a discrete valuation ring.
Hence Theorems \ref{theorem-Fone} and \ref{thm-s1} are consequence of the following two results.

\begin{thm} \label{theorem-Fone-DVR}
 Every object $(M,U)$ in $\mathcal{F}_1(\SigCyc(R))$  is isomorphic to a~direct sum of pickets
 of type $P_{n-1}^n$ or $P_n^n$.
 The direct sum decomposition is unique up to reordering and isomorphy of the summands.
\end{thm}

\begin{thm} \label{mainR}
 An object $(M,U)$ in $\mathcal{S}_1(\SigCyc(R))$  is isomorphic to a~direct sum of pickets
 of type $P_0^n$ or $P_1^n$ if and only if $M/U\in\SigCyc(R)$.
 The direct sum decomposition is unique up to reordering and isomorphy of the summands.
\end{thm}

\section{Proof of Theorem \ref{theorem-Fone-DVR}}
\label{seq-proof-thm1}

We first show the result about $\mathcal F_1$.

\subsection{Nice filtrations}

Let $(M,U)$ be an object in $\mathcal F_1(\SigCyc(R))$, so
the $R$-module $M$ is a direct sum of cyclic $R$-modules and $U\subseteq M$
a submodule such that the factor $M/U$ is semisimple.
In a decomposition of $M$,
combine summands of the same length to obtain the isotypic decomposition
$$M=\bigoplus_{n\in\mathbb N} T_n$$
where $T_n$ is the direct sum of all summands of length $n$ in the given direct sum
decomposition of $M$. We put  $T_0=0$.

For each $n\in\mathbb N$,
write $T_n=\bigoplus_{\kappa_n} R/(p^n)$ as a direct sum
of $\kappa_n$ copies of $R/(p^n)$ where the index set
$\kappa_n$ is a cardinal number.

  We will use the associated filtration,
  $$M=\bigcup_{n\in\mathbb N} M_n,\qquad\text{where}\qquad
  M_n=\bigoplus_{\ell\leq n} T_\ell.$$
  Since $M$ has a filtration, so does the submodule $U$:
  $$U=\bigcup_{n\in\mathbb N}U_n,\qquad\text{where}\qquad
  U_n=U\cap M_n.$$

  Then each pair $(M_n,U_n)$ is also in $\mathcal F_1(\SigCyc(R))$.
  In the following sections we show that the filtration is nice
  in the sense that
  $$(*)\qquad (M_n,U_n)=(M_{n-1},U_{n-1})\oplus D\quad\text{where}\quad
  D\cong (P_n^n)^{\mu_n^n}\oplus (P_{n-1}^n)^{\mu_{n-1}^n}$$
  holds.  From formula $(*)$ we can derive
  the multiplicities of the pickets in the direct sum decomposition:

    Since $M_n=M_{n-1}\oplus T_n$ and $M_{n-1}+pM_n=M_{n-1}\oplus pT_n$, the vector space
    $\frac{M_n}{M_{n-1}+pM_n}$ has dimension $\kappa_n$.  Considering the chain
    $$M_{n-1}+pM_n\;\subseteq\; M_{n-1}+U_n\;\subseteq\; M_n,$$
    we obtain $\kappa_n=\mu_n^n+\mu_{n-1}^n$ where
    $$\mu_n^n=\dim\frac{M_{n-1}+U_n}{M_{n-1}+pM_n}
    \quad\text{and}\quad \mu_{n-1}^n=\dim\frac{M_n}{M_{n-1}+U_n}.$$

  Thus, the embedding $(M,U)$ is the union, and hence the direct sum, as follows,
  $$(M,U)\;=\;\bigcup_{n\in\mathbb N}(M_n,U_n)\;\cong\;\bigoplus_{n\in\mathbb N}
  \left[(P_n^n)^{\mu_n^n}\oplus(P_{n-1}^n)^{\mu_{n-1}^n}\right],$$
  proving the existence of the decomposition in the theorem.
  Since each picket has a local endomorphism ring,
  uniqueness follows from the Krull-Remak-Schmidt-Azumaya theorem, see \cite{azu,fac}.

  \begin{rem}
    Note that the assumption $(M,U)\in\mathcal F_1(\SigCyc(R))$ is necessary for $(*)$:
    For example consider $M=R/(p)\oplus R/(p^2)$ and $U$ generated
    by $(1,p)\in G$.  Then $(M,U)\cong P^1_1\oplus P^2_0$.
    If $M_1$ is the first summand of $M$ then $U_1=M_1\cap U=0$, so
    $(M_1,U_1)\cong P^1_0$.  Hence $(M_1,U_1)$ is not isomorphic to a direct summand of
    $(M_2,U_2)=(M,U)$.
  \end{rem}

\subsection{The setup for the induction}

We assume the notation from the previous section.
It remains to show that property $(*)$ holds.  For this,
we construct successively
minimal generating sets for the modules $M_n$ in the filtration.  Such a generating set
has the form $C_n\dotcup C_n'$ where
  $$M_n=\bigoplus_{g\in C_n\dotcup C_n'} gR,\qquad U_n=\bigoplus_{g\in C_n} gR\oplus
\bigoplus_{g\in C_n'} pgR.$$
In particular,
$$(M_n,U_n)\quad \cong\quad  \bigoplus_{g\in C_n}(gR,gR)\quad\oplus\quad
\bigoplus_{g\in C_n'} (gR,p\,gR).$$
As a byproduct, we obtain sets
$$B_n=\{g+pM_n:g\in C_n\},\qquad B_n'=\{g+pM_n:g\in C_n'\}$$
such that $B_n$ is a basis for $\frac{U_n}{pM_n}$ and
$B_n\dotcup B_n'$ a basis for $\frac{M_n}{pM_n}$.

In the case where $n=1$, the module $M_1$ is a vector space over $R/(p)$ so
we can take for $C_1=B_1$ a basis for the subspace $U_1=M_1\cap U$ and extend it to a basis
$C_1\dotcup C_1'=B_1\dotcup B_1'$ for $M_1$.

\subsection{Constructing bases}

Assume now that  minimal generating sets $C_{n-1}$, $C_{n-1}'$
and bases $B_{n-1}$, $B_{n-1}'$ have been constructed.
Put
$$\tilde B_{n-1}=\{g+pM_n:g\in C_{n-1}\},\qquad \tilde B_{n-1}'=\{g+pM_n:g\in C_{n-1}'\},$$
then $\tilde B_{n-1}$ is in bijection to $B_{n-1}$ under the isomorphism
$\frac{U_{n-1}+pM_n}{pM_n}\cong \frac{U_{n-1}}{U_{n-1}\cap pM_n}=\frac{U_{n-1}}{pM_{n-1}}$ and hence a
basis, and
$\tilde B_{n-1}\cup \tilde B_{n-1}'$ is in bijection to $B_{n-1}\cup B_{n-1}'$ under the
isomorphism $\frac{M_{n-1}+pM_n}{pM_n}\cong\frac{M_{n-1}}{M_{n-1}\cap pM_n}=\frac{M_{n-1}}{pM_{n-1}}$
and hence a basis.
Consider the following diagram.

$$\xymatrixrowsep{1pc}
\xymatrix{ & \frac{M_n}{pM_n} \ar@{-}[d] & \\
  & \frac{U_n+M_{n-1}}{pM_n} \ar@{-}[dl] \ar@{-}[dr]  & \\
  \frac{U_n}{pM_n} \ar@{-}[dr] & & \frac{M_{n-1}+pM_n}{pM_n} \ar@{-}[dl] \\
  & \frac{U_{n-1}+pM_n}{pM_n} \ar@{-}[d] &\\
  & 0 & }$$

Note that
\begin{eqnarray*}U_n\cap (M_{n-1}+pM_n) & = & (U_n\cap M_{n-1})+pM_n\\
  & = & (U\cap M_n\cap M_{n-1})+pM_n\\
  & = & U_{n-1}+pM_n.
\end{eqnarray*}

For the first equation, recall $pM_n\leq U_n$ since we are dealing
with objects in $\mathcal F_1$. So, the space in the fourth row in the diagram is the intersection of the spaces on the left
and on the right, while the space in the second row near the top is their sum.
Extend the basis $\tilde B_{n-1}$ for $\frac{U_{n-1}+pM_n}{pM_n}$ $($in the fourth row) to a basis
$B_n$ for $\frac{U_n}{pM_n}$ $($on the left).
Recall that $\tilde B_{n-1}\dotcup \tilde B_{n-1}'$ is a basis for $\frac{M_{n-1}+pM_n}{pM_n}$.
Then $B_n\dotcup \tilde B_{n-1}'$ is a linearly
independent set, hence a basis for $\frac{U_n+M_{n-1}}{pM_n}$.
Extend $\tilde B_{n-1}'$ to a set $B_n'$ so that $B_n\dotcup B_n'$ is a basis for $\frac{M_n}{pM_n}$
$($at the top).

\subsection{A lifting lemma}

We will use the following result, which may be well known,
to lift the basis $B_n\dotcup B_n'$
to a set of generators for $M_n$.

\begin{lem}
  Suppose $T$ is a direct sum of copies of $R/(p^n)$.
  Let $B=\{b_i:i\in I\}$ be a linearly independent set in $T/pT$
  and $C=\{c_i:i\in I\}$ any set of liftings of $B$ to $T$ satisfying $b_i=c_i+pT$.
  Then the submodule $M$ of $T$
  generated by the elements $c_i$ is the direct sum
  $$M=\bigoplus_{i\in I}c_i \mathbb Z.$$
  Moreover, if $B$ is a basis for $T/pT$, then $T=\bigoplus_{i\in I} c_i\mathbb Z$.
\label{lem-GT}\end{lem}

\begin{proof}
  We show by induction that for each $m=n-1,\ldots,0$ the following statement holds:

  $$(P_m) \quad p^mM=\bigoplus_{i\in I} p^mc_iR; \quad\text{if $B$ is a basis then}\quad
  p^mT=p^mM.$$

  We will use that the map
  $$\varphi_m: \frac{T}{pT}\to\frac{p^mT}{p^{m+1}T}, \;x+pT\mapsto p^mx+p^{m+1}T$$
  is an isomorphism of vector spaces.  In particular, the statement $(P_m)$ holds
  for $m=n-1$.

  We deal with the induction step $m+1\to m$.
  Let $B'=\{b_i'=\varphi_m(b_i)  :i\in I\}$.
  Then $B'$ is a linearly independent set and, if $B$ is a basis, then so is $B'$.
  The elements $b_i'$ satisfy $b_i'=p^mc_i+p^{m+1}T$ for $i\in I$.

  The elements $p^mc_i$ form a generating set for $p^mM=\sum_{i\in I}p^mc_iR$.

  To show that the sum is a direct sum, suppose $J\subseteq I$ is a~finite subset
  and $0=\sum_{i\in J} a_ip^mc_i$ where
  $a_i\in R$.   Then in $p^mT/p^{m+1}T$, we have
  $0=\sum a_ib_i'$.  Since  $B'$ is a linearly independent set,
  for all $i\in J$ there exist $a_i'\in R$ such that $a_i=a_i'p$.
  Since $\sum a_i'p^{m+1}c_i=0$ we can use the induction hypothesis $(P_{m+1})$ to conclude
  that $a_i'p^{m+1}c_i=0$ for all $i\in J$, hence each $a_ip^mc_i=0$, so
  the sum is a direct sum.

  It remains to show that if $B$ is a basis then $p^mT=p^mM$.
  Let $x\in p^mT$, then $x+p^{m+1}T=\sum a_ib_i'$ with $a_i\in R$.
  Hence $x-\sum a_ip^mc_i$ is in $p^{m+1}T=p^{m+1}M$, by induction hypothesis $(P_{m+1})$,
  so there are
  $a_i'\in R$ with $x-\sum a_ip^mc_i=\sum a_i'p^{m+1}c_i$.
  This shows that $x\in p^mM$.
\end{proof}

\begin{prop}
  \label{prop-lifting}
  Suppose we have the following setup:
  \begin{itemize}
    \item $F=F'\oplus T$ where $F'$ is $p^{n-1}$-bounded and $T$ is a direct
      sum of copies of $R/(p^n)$.
    \item $B'$ is a basis for
        $F'/pF'$ and $B=B'\dotcup B''$ a linearly independent subset in $F/pF$ extending $B'$.
    \item $C'$ is a set of elements in $F'$ corresponding to the elements in $B'$,
      and $C''$ a set of liftings for the elements in $B''$ to $F$.
  \end{itemize}
  Then the sum $M=\sum_{c\in C'\dotcup C''} c R$ is a direct sum.
  Moreover, if $B$ is a basis for $F/pF$ then $F=M$.
\end{prop}

\begin{proof}
  We show the result by induction on $n$.  The case $n=1$ is clear.
  Suppose $n>1$.  By induction hypothesis, we have that the sum
  $\sum_{c\in C'} cR$ is the direct sum $F'=\bigoplus_{c\in C'}cR$.
  Write each $c\in C''$ as a pair $c=(c_1,c_2)$ with $c_1\in F'$ and $c_2\in T$,
  and similarly, each $b\in B''$ as a pair $b=(b_1,b_2)$ with $b_1\in F'/pF'$
  and $b_2\in T/pT$.  If $b=c+pF$ then $(b_1,b_2)=(c_1+pF',c_2+pT')$.

  By Lemma \ref{lem-GT}, the sum $\sum_{(c_1,c_2)\in C''}c_2R$ is a direct sum,
  hence so is $M=\sum_{c\in C'\dotcup C''}cR$.
  Moreover, if $B$ is a basis for $F/pF$ then $B''$ is a basis for $T/pT$,
  hence we obtain from Lemma \ref{lem-GT} that $T=\bigoplus_{(c_1,c_2)\in C''}c_2R$,
  and it follows that $F=F'\oplus T=\bigoplus_{c\in C'\dotcup C''} cR$.
\end{proof}

\subsection{Constructing generators}

Recall that $B_n$ is a basis for $U_n/pM_n$ and that
the set $C_{n-1}$ consists of generators in $M_{n-1}$, corresponding to
basis elements in $\tilde B_{n-1}=\{g+pM_n:g\in C_{n-1} \}$.
  Extend $C_{n-1}$ by choosing for each basis element
$b\in B_n\setminus \tilde B_{n-1}$ an element $g\in M_n$ which satisfies $b=g+pM_n$, this is the set $C_n$.
Similarly, extend $C_{n-1}'$ by choosing elements $g\in M_n$ for the basis elements
$b\in B_n'\setminus \tilde B_{n-1}'$ to obtain the set $C_n'$.

We obtain from Proposition~\ref{prop-lifting} that
$$(**)\qquad M_n=\bigoplus_{g\in C_n\dotcup C_n'} gR.$$

It remains to verify $U_n=\bigoplus_{g\in C_n} gR\oplus \bigoplus_{g\in C_n'} pgR$.
  The right hand side is contained in the left hand side:
  Since $U_n\supset pM_n$, we have $pgR\subset U_n$ for $g\in C_n'$.
  Let $g\in C_n$.  For $\bar g\in U_n/pM_n$ there is $h\in U_n$
  such that $\bar h=\bar g$, hence $g-h\in p{M_n}\subset U_n$.  Hence $g\in U_n$.

  The left hand side is contained in the right hand side:
  Let $x\in U_n$, then $\bar x\in U_n/pM_n$ is a linear combination
  $\bar x=\sum a_i b_i$ where $a_i\in R$, $b_i\in B_n$.
  Suppose $g_i\in C_n$, $\bar g_i=b_i$, so $x-\sum a_ig_i\in pM_n$,
  hence by $(**)$ there are $k_i\in C_n$, $\ell_i\in C_n'$,
  $d_i,e_i\in R$
  such that $x-\sum a_ig_i=p(\sum d_ik_i+\sum e_i\ell_i)$, hence
  $$x=[\sum a_ig_i+p\sum d_ik_i]+[p\sum e_i\ell_i]
  \in \bigoplus_{g\in C_n}gR\oplus \bigoplus_{g\in C_n'}pgR.$$

  To finish the proof of the theorem, note that
  $$(M_n,U_n)\;=\;\bigoplus_{g\in C_n}(gR,gR)
  \oplus\bigoplus_{g\in C_n'}(gR, pgR)$$
  extends the previous decomposition for $(M_{n-1},U_{n-1})$.

  \section{Proof of Theorem \ref{mainR}} \label{sec-proof3}

  Now we deal with $\mathcal S_1$.
  We split the proof of Theorem \ref{mainR} into several lemmata. The next two Lemmata are adaptations of Lemmata 10 and 11 in \cite{kap}.

\begin{lem}
 Let $M$ be a primary $R$-module and $N$ a pure submodule, let $x\in M[p],x\notin N$ with $h(x)=k<\infty$ and $\forall a\in N[p] (h(x+a)\leq h(x))$. Write $x=p^k y$ and let $L=N+Ry$. Then $L=N\oplus Ry$ and $L$ is pure in $M$. \label{3.2}
\end{lem}

\begin{proof}
 $N\cap Ry=0$ because any nonzero submodule of $Ry$ contains $x$, whereas $x\notin N$. We prove that $L$ is pure in $M$. Let $w\in L[p]$. We show that $w$ has the same height in $L$ as in $M$. We can assume that $w= x+a$, where $a\in N[p]$. We have $h_N (a)\leq h_L (a)\leq h_M (w)=h_N (a)$, because $N$ is pure in $M$. Now $h_L (x)=h_M (x)=k$ because $x=p^k y$ and $y\in L$.  If $h(a)\neq k$, then $h(x+a)=\min\{k,h(a)\}$ in both, $L$ and $M$. If $h(a)=k$, then $h(x+a)\geq k$, but $h(x)=k$ is by hypothesis maximal among all $h(x+a)=h(w)$, so again $h_L (w)=h_M (w)$. It follows that $L$ is pure in $M$.
\end{proof}

\begin{lem}
Let $M$ be a~primary $R$-module, let $P,Q$ be submodules
with $P\subseteq Q\subseteq M[p]$, and suppose $Q$ is of bounded height. Let $\{x_i|i\in I\}$ be a~pure set satisfying $(\bigoplus_{i\in I} Rx_i) \cap M[p]=P$. Then $\{x_i|i\in I\}$ can be enlarged to a~pure set $\{y_j|j\in J\}$ satisfying $(\bigoplus_{j\in J} Ry_j)\cap M[p]=Q$. \label{lem-kap-pure}
\end{lem}

\begin{proof}
 Consider pure sets $\{y_j|j\in J\}$ which contain $\{x_i|i\in I\}$ and satisfy $(\bigoplus_{j\in J} Ry_j)\cap M[p] \subseteq Q$. By Kuratowski-Zorn's lemma we may pick a maximal one among these sets; we denote it again by $\{y_j|j\in J\}$. Assume $z\in Q$, but $z\notin ( \bigoplus_{j\in J} Ry_j)\cap M[p]$. Since $\{z+a|a\in ( \bigoplus_{j\in J} Ry_j)\cap M[p]\}\subseteq Q$ and $Q$ is of bounded height, there is an $a\in ( \bigoplus_{j\in J} Ry_j)\cap M[p]$ such that $x=z+a$ is of maximal height $h_M(x)=k$. Write $x=p^k y$, for $y\in M$. Then by Lemma \ref{3.2} the enlarged set $\{y\}\cup\{y_j|j\in J\}$ is again pure. Since $(Ry\oplus\bigoplus_{j\in J} Ry_j)\cap M[p]\subseteq Q$, we have contradicted the maximality of $\{y_j|j\in J\}$.
\end{proof}

We adapt two theorems by Kulikov, originally formulated for abelian $p$-groups.

\begin{thm}\label{kul1}
A primary $R$-module $M$ is a direct sum of cyclic modules if and only if $M[p]$ is the union of an~ascending chain of height-bounded submodules.
\end{thm}

\begin{proof}
 The necessity is clear. To prove the sufficiency let $M[p]=\bigcup_{i\geq 1}P_i$, where
 $P_1\subseteq P_2\subseteq P_3\subseteq\ldots$ and each submodule $P_i$ is height-bounded. We construct pure sets $X_i$ such that $X_1\subseteq X_2\subseteq X_3\subseteq\ldots$ and $\langle X_i\rangle\cap M[p]=P_i$ for $i\in\mathbb{N}$.  Suppose $X_1,\ldots,X_n$ have already been constructed. Lemma \ref{lem-kap-pure} enables us to construct $X_{n+1}$, such that $X_{n}$ is it's direct summand. Let $X=\bigcup_{i\geq 1}X_i$. Then the submodule $\langle X\rangle$ generated by $X$ contains $M[p]$, is a~direct sum of cyclic modules and is pure in $M$.

 We prove $\langle X\rangle=M$ by induction. Suppose $\langle X\rangle$ contains all elements of order $\leq p^n$, and let $x\in M$ be of order $p^{n+1}$. Then $px\in\langle X\rangle$, and by purity there is an $y\in\langle X\rangle$ with $px=py$. So $x-y\in M[p]\subseteq\langle X\rangle$, hence $x\in\langle X\rangle$. It follows that $M$ is a direct sum of cyclic modules.
\end{proof}

\begin{thm} \label{kul2}
 Let an~$R$-module $M$ be a direct sum of cyclic modules. Then any submodule $N$ is again a direct sum of cyclic modules.
\end{thm}

\begin{proof}
 \cite[Theorem 8.5]{kt}
\end{proof}

\begin{lem}
  Let $(M,U)\in\mathcal{S}_1(\SigCyc(R))$. There exists a~pure submodule $N$ of $M$ such that $N[p]=U$.
  Moreover $(N,U)$ is isomorphic to a~direct sum of pickets of type $P_1^n$.
\label{NpU}
\end{lem}

\begin{proof}
Since $M\in\SigCyc(R)$, we have  $M[p]=\bigcup_{i\geq 1}P_i$, where $P_1\subseteq P_2\subseteq \ldots $ and each submodule $P_i$ is height-bounded, by Theorem \ref{kul1}. Let $U_i=U\cap P_i$. Then $U=\bigcup_{i\geq 1}U_i$, where $U_1\subseteq U_2\subseteq\ldots $, and each subgroup $U_i$ is of bounded height.

We are going to construct pure subsets
$X_1\subseteq X_2\subseteq\ldots$ such that $\langle X_i\rangle\cap P=U_i$.
Suppose that $X_1,\ldots, X_n$ are constructed.
We apply Lemma \ref{lem-kap-pure} to $P=\langle X_n\rangle\cap M[p]=U_n$,
$Q=U_{n+1}$ and get a~pure subset $X_{n+1}$ such that $\langle X_{n+1}\rangle\cap M[p]=U_{n+1}$.
Let $X=\bigcup_{i\geq 1}X_i$. Then $N=\langle X\rangle$ is a~pure submodule of $M$ with $N[p]=U$.

Since $M\in\SigCyc(R)$, we have $N\in\SigCyc$, by Theorem \ref{kul2}. It follows that $(N,U)$ is a~direct sum of pickets of the form $P_1^n$, because $N[p]=U$.
\end{proof}

\begin{lem}\label{lem-directsummand}
  Let $(M,U)\in\mathcal{S}_1(\SigCyc(R))$ such that $M/U\in\SigCyc(R)$.
  Let $N$ be a~pure submodule of $M$ such that $N[p]=U$. Then $N$ is a~direct summand of $M$.
\end{lem}

\begin{proof}
We know from Lemma \ref{NpU} that $M$ has a pure submodule $N$ whose socle is $N[p]=U$.
We show that $M/N\in\SigCyc(R)$. Consider the epimorphism
$$\pi:M/U\longrightarrow M/N,\qquad x+U\longmapsto x+N$$

We claim that $\pi(M[p]/U)=(M/N)[p]$. Let $x\in M[p]$, then $p\pi(x+U)=px+N=N$, hence $x+N\in (M/N)[p]$. Conversely, if $x+N\in (M/N)[p]$, then $px\in N$ and by the purity of $N$ there is a $y\in N$ with $px=py$, hence $x-y\in M[p]$. It follows that $x+N=x-y+N\in\pi(M[p]/U)$.

Since $M/U$ is a direct sum of cyclic modules, by Theorem \ref{kul1} we have $M[p]/U=\bigcup_{i\geq 1}A_i$, where $A_1\subseteq A_2\subseteq\ldots$ is an ascending chain of height-bounded submodules. Now let $\pi(M[p]/U)=\bigcup_{i\geq 1}\pi(A_i), \quad $

The height in $M/U$ of an element $x+U\in M[p]/U$ is greater or equal to the height of $\pi(x+U)$ in $M/N$: Let $h(x+U)=k$ for some $k\in\mathbb{N}$, then $x+U=p^k y+U$ for some $y\in M$ and $p^{k+1} y\in U\subseteq N$, hence $h(\pi(x+U))=h(p^k y+N)\leq k$.
It follows that each $\pi(A_i)$ is height-bounded, so again by Theorem \ref{kul1} $M/N\in\SigCyc(R)$.
Therefore, by \cite[Theorem 10.2, $(2)\Leftrightarrow(3)$]{kt}, $N$ is a~direct summand of $M$, because $N$ is pure in $M$. 
\end{proof}

We are ready to prove Theorem \ref{mainR}.

\begin{proof}
  Let $(M,U)\in\mathcal{S}_1(\SigCyc(R))$ such that $M/U\in\SigCyc(R)$.
  Let $N$ be a~pure submodule of $M$ such that $N[p]=U$.
  Such a~submodule exists by Lemma \ref{NpU}.
  Then $(N,U)$ is isomorphic to a~direct sum of pickets of type $P_1^n$.

  By Lemma \ref{lem-directsummand}, the module $N$ is a~direct summand of $M$,
  hence $M=N\oplus N'$ for some submodule $N'$ of $M$.
  The submodule $N'$ is a~direct sum of cyclic groups, by Theorem \ref{kul2}. Therefore $(M,U)=(N,U)\oplus (N',0)$ and $(N',0)$ is isomorphic to a~direct sum of pickets of type $P_0^n$.

  As each picket has a local endomorphism ring,
  uniqueness follows from the Krull-Remak-Schmidt-Azumaya theorem, see \cite{azu,fac}.

  Regarding the converse, whenever $(M,U)$ is isomorphic to a direct sum of pickets,
    then the factor $M/U$ is in $\SigCyc(R)$.
\end{proof}

\subsection*{Acknowledgment} The first named author wants to thank Stanis\l aw Kasjan for fruitful discussions.

\end{document}